\documentclass[10pt]{amsart}
\usepackage{amsmath}
\usepackage{amsfonts}
\usepackage{amssymb}
\usepackage{amsthm}
\usepackage{url}
\usepackage{dsfont}
\usepackage{graphicx}
\usepackage{caption}
\usepackage{subcaption}
\usepackage{comment} 
\usepackage{stmaryrd}
\usepackage{hyperref}
\usepackage{todonotes}
\usepackage{color}
\usepackage{enumerate}

\numberwithin{equation}{section}

\newtheorem{proposition}{Proposition}[section]
\newtheorem{lemma}[proposition]{Lemma}
\newtheorem{theorem}[proposition]{Theorem}

\theoremstyle{definition}
\newtheorem{remark}[proposition]{Remark}
\newtheorem{definition}[proposition]{Definition}

\DeclareMathOperator{\tr}{tr}
\DeclareMathOperator{\gr}{gr}

\DeclareMathOperator{\Aut}{Aut}

\DeclareMathOperator{\Proj}{Proj}

\DeclareMathOperator{\DF}{DF}

\DeclareMathOperator{\Lie}{Lie}
\DeclareMathOperator{\opF}{F}

\newcommand{\C}{\mathbb{C}}

\newcommand{\Q}{\mathbb{Q}}

\renewcommand{\epsilon}{\varepsilon}

\renewcommand{\L}{\mathcal{L}}
\newcommand{\X}{\mathcal{X}}
\newcommand{\Y}{\mathcal{Y}}
\renewcommand{\H}{\mathcal{H}}
\newcommand{\scZ}{\mathcal{Z}}

\title[K-semistability of optimal degenerations]{K-semistability of optimal degenerations}

\author[Ruadha\'i Dervan]{Ruadha\'i Dervan}

\address{Ruadha\'i Dervan, DPMMS, Centre for Mathematical Sciences, Wilberforce Road, Cambridge CB3 0WB, United Kingdom}\email{R.Dervan@dpmms.cam.ac.uk}
\begin{document}

\begin{abstract} K-polystability of a polarised variety is an algebro-geometric notion conjecturally equivalent to the existence of a constant scalar curvature K\"ahler metric. When a variety is K-unstable, it is expected to admit a ``most destabilising'' degeneration. In this note we show that if such a degeneration exists, then the limiting scheme is itself relatively K-semistable.
\end{abstract}

\maketitle

\section{Introduction}

Given a smooth projective variety $X$ endowed with an ample line bundle $L$, one of the central questions in K\"ahler geometry is to understand whether or not the K\"ahler class $c_1(L)$ admits a K\"ahler metric of constant scalar curvature. The Yau-Tian-Donaldson conjecture states that the existence of such a metric should be equivalent to the purely algebro-geometric notion of K-polystability \cite{STY,GT,SD2}. This conjecture has been completely solved when $L=-K_X$ is the anticanonical class, so that $X$ is Fano and the metric is necessarily K\"ahler-Einstein \cite{GT,RB,CDS}, but is open in general. When $(X,L)$ is given the extra structure of a rational vector field $v$, one can also ask for $(X,L)$ to be K-polystable relative to $v$ \cite{gabor-blms}, which is then conjecturally equivalent to the existence of an extremal metric in $c_1(L)$.

By definition, K-polystability means that for all $\C^*$-degenerations of $(X,L)$ a certain weight called the Donaldson-Futaki invariant is non-negative, and equal to zero if and only if the $\C^*$-action fixes $(X,L)$. These $\C^*$-degenerations are called test configurations. Thus if $(X,L)$ is not K-polystable, then it admits a degeneration to another scheme $(Y,H)$ with non-positive Donaldson-Futaki invariant. Such a test configuration is far from unique when it exists, and so it is natural to ask if there is a test configuration which is most responsible for the instability of $(X,L)$. 

A concrete way of interpreting this question is to ask for a test configuration for which the \emph{normalised} Donaldson-Futaki invariant is minimised, namely for which the Donaldson-Futaki invariant divided by the $L^2$-norm is minimised \cite{SD2}. We call such a test configuration an optimal degeneration when it exists. Existence has been proven by Chen-Sun-Wang on polarised manifolds provided the Calabi flow satisfies a strong curvature boundedness assumption \cite{CSW}, and a weak version of existence has been proven in the toric case by Sz\'ekelyhidi again under assumptions on the Calabi flow \cite{GS-toric}. 

Note that if $(Y,H)$ is the central fibre of an optimal degeneration, then it naturally admits a $\C^*$-action induced from the structure of the test configuration, with corresponding vector field which we denote $v$. The purpose of this short note is to prove the following:

\begin{theorem}\label{intromainthm} Suppose $(Y,H)$ is the central fibre of an optimal degeneration for $(X,L)$, with $\C^*$-action induced by a vector field $v$. Then $(Y,H)$ is K-semistable relative to $v$. \end{theorem}

Our proof is completely algebro-geometric, and thus holds for arbitrary varieties or schemes $(Y,H)$. In the K-semistable (rather than relatively K-semistable) case, the proof simplifies and is due to Li-Wang-Xu in the Fano setting \cite{lwx2}. In this case, the Theorem states that if $(X,L)$ degenerates to $(Y,H)$ with vanishing Donaldson-Futaki invariant, then $(Y,H)$ must be K-semistable. Our main technical tool is a result of Li-Xu and Li-Wang-Xu relating degenerations of $(X,L)$ to those of $(Y,H)$, which uses ideas from the theory of toric degenerations \cite{LX,lwx2}.

An analytic motivation to study optimal degenerations is the following inequality due to Donaldson.

\begin{theorem}\cite{SD} We have $$\inf_{\omega \in c_1(L)} \|S(\omega) - \hat S\|_2 \geq - \inf_{(\X,\L)} \frac{\DF(\X,\L)}{\|(\X,\L)\|_2},$$ where the infimum on the right hand side is taken over all test configurations for $(X,L)$. Here $S(\omega)$ denotes the scalar curvature, $\hat S$ denotes the average scalar curvature, $\DF(\X,\L)$ denotes the Donaldson-Futaki invariant, $\|(\X,\L)\|_2$ denotes the $L^2$-norm of the test configuration $(\X,\L)$ and $\|S(\omega) - \hat S\|_2^2 = \int_X (S(\omega) - \hat S)^2\omega^n$ is the $L^2$-norm of the average scalar curvature.  \end{theorem}

Donaldson conjectures that the two infima actually coincide. Thus conjecturally an optimal degeneration is precisely a test configuration which realises the infimum of the norm squared of the scalar curvature, i.e. the Calabi functional. Our main result then says that a test configuration realising the infimum of the Calabi functional must be a degeneration to a relatively K-semistable scheme. 

K-stability is motivated by, and is closely related to, Geometric Invariant Theory. From this point of view, an optimal degeneration is the direct analogue of Kempf's most destabilising one-parameter subgroups \cite{GK}. This perspective also leads one to expect that Theorem \ref{intromainthm} is essentially sharp, i.e. $(Y,H)$ should not be K-polystable relative to $v$ in general, but instead $(Y,H)$ should be a deformation of a K-polystable scheme which, if it were a smooth variety, should admit an extremal metric. In the situation of Chen-Sun-Wang, it is shown that $(Y,H)$ is a small deformation of another polarised manifold which admits an extremal metric \cite{CSW}. Even then it seems to be new that $(Y,H)$ is relatively K-semistable, though it might be possible to obtain this from work of Clarke-Tipler \cite{CT}.

A conjecture of Donaldson states that an optimal degeneration should, in some sense, split up a variety into relatively log K-polystable pieces  \cite{GS-toric}, where by log K-polystability we include the data of a divisor. One concrete way of interpreting this is to consider this process in three steps. Firstly, there should be an optimal degeneration to a relatively K-semistable demi-normal variety $(Y,H)$ (see \cite[Definition 5.1]{kollar-singularities} for the definition of demi-normality), with relative K-semistability following from Theorem \ref{intromainthm}. In general, $Y$ will not be irreducible, but one should as the next step consider the normalisation $(Y^{\nu}, H^{\nu})$ together with the conductor divisor $D$ on $Y^{\nu}$ induced from the normalisation process. Then each test configuration for $(Y,H)$ induces one for $(Y^{\nu}, H^{\nu})$ and it is not hard to see from the intersection-theoretic definition that the Donaldson-Futaki invariant for $(Y,H)$ equals the log Donaldson-Futaki invariant for $(Y^{\nu}, H^{\nu})$ with respect to $D$ \cite{OS}. One can further show that the $L^2$-norm is unchanged by the normalisation process (see for example \cite[Theorem 3.14]{BHJ}, or one could use the behaviour of Euler characteristics under normalisation following \cite[p470]{SD}), and hence the inner product of two vector fields is also unchanged by normalisation by an identical argument to \cite[p80]{gabor-blms}, namely writing the inner product as a difference of two norms. In summary, $(Y^{\nu}, H^{\nu})$ together with $D$ is relatively log K-semistable with respect to those test configurations induced from $(Y,H)$. Unfortunately, this is not enough to show relative log K-semistability in general as there are more test configurations to consider. However, this does seem to strongly suggest that this is the geometric manner in which an arbitrary variety breaks up into relatively log K-semistable pieces. The final step arises from the expectation that each of these pieces should admit a relatively log K-polystable degeneration.

Lastly, we remark that it may be slightly too optimistic to expect that an optimal degeneration exists in general. Instead what seems most likely in complete generality is that $(X,L)$ admits a most destabilising filtration  in the sense of Witt Nystr\"om \cite[Section 1.3]{DWN} and Sz\'ekelyhidi \cite[Section 3]{filt}. If $$\C = F_0 \subset F_1 R \subset F_2R \subset \hdots \subset R$$ is a filtration of the coordinate ring $R = \oplus_{k \geq 0} H^0(X,L^k)$ of $(X,L)$, then one has an associated graded algebra $$\gr(F) = \oplus_{j \geq 0} (F_jR)/(F_{j-1}R).$$ When this graded algebra is finitely generated, Sz\'ekelyhidi shows that $\Proj_{\C} \gr(F)$ corresponds to the central fibre of a test configuration for $(X,L)$ \cite[Section 3.1]{filt}. Even when $\gr(F)$ is not finitely generated, the homogeneous spectrum $\Proj_{\C} \gr(F)$ \cite[Definition 00JN]{stacks-project} still produces a scheme \cite[Lemma 01MB]{stacks-project}, which will no longer be Noetherian. It is interesting to ask what the geometry of the limiting object $\Proj_{\C} \gr(F)$ is in this case, and in particular whether or not there is some sense in which it is relatively K-semistable. Such non-Noetherian schemes are speculated to be relevant for the compactification of moduli spaces of K-polystable varieties \cite{DN}. Lastly, we remark that in the Fano case, optimal degenerations \emph{are} known to exist if one replaces the normalised Donaldson-Futaki invariant with a slightly different numerical invariant \cite{CSW,DS}.

\vspace{4mm} \noindent {\bf Funding:} This work was supported by the \emph{Agence nationale de la recherche} grant \emph{GRACK}.

\vspace{4mm} \noindent {\bf Acknowledgements:} I thank Chi Li, Lars Sektnan and G\'abor Sz\'ekelyhidi for useful discussions. I also thank the referee for their helpful comments.

\section{K-stability} We work throughout over the complex numbers. Fix a polarised scheme $(X,L)$, so that $X$ is projective and $L$ is ample. 

\begin{definition}\cite{SD2} A \emph{test configuration} for $(X,L)$ is a scheme $\X$ with
\begin{enumerate}[(i)]
\item a surjective flat morphism $\pi: \X\to\C$,
\item  a relatively ample line bundle $\L$, 
\item a $\C^*$-action $\alpha$ commuting with $\pi$ and lifting to $\L$,
\end{enumerate}
such that $\pi^{-1}(t) \cong (X,L^{\otimes r})$ for all $t\neq 0$. We call $r$ the \emph{exponent} of $(\X,\L)$. We say that $(\X,\L)$ is a \emph{product} if $\X_0\cong X$.

\end{definition}

Each test configuration admits several important numerical invariants, defined using the induced $\C^*$-action on the central fibre $(\X_0,\L_0)$. Denote for $k \gg 0$ the Hilbert polynomial of $(\X_0,\L_0)$ by $$h(k) = \dim H^0(\X_0,\L_0^k) = a_0k^{n} + a_1k^{n-1} + O(k^{n-2}),$$ where $n=\dim X$. If the $\C^*$-action $\alpha$ on $H^0(\X_0,\L_0^k)$ has infinitesimal generator $A_k$, we define the \emph{total weight} of the $\C^*$-action for $k \gg 0$ to be $$w(k) =\tr (A_k)= b_0k^{n+1} + b_1 k^n + O(k^{n-1}).$$ This is simply the weight of the $\C^*$-action on the determinant of $H^0(\X_0,\L_0^k)$.

Suppose that $(\X,\L)$ admits an additional $\C^*$-action $\beta$ which fixes each fibre of $\pi$ and commutes with $\alpha$, hence induces a $\C^*$-action on $H^0(\X_0,\L_0^k)$ with infinitesimal generator $B_k$. We then say that $(\X,\L)$ is $\beta$\emph{-equivariant} and  denote $$d(k) =  \tr(A_k B_k)= d_0k^{n+2} +O(k^{n+1}).$$ It follows from their construction that the $a_i,b_i$ and $d_i$ are rational numbers \cite{SD}.

\begin{definition}\cite{SD,gabor-blms} We define the
\begin{enumerate}[(i)]
\item  \emph{Donaldson-Futaki invariant} of $(\X,\L)$ to be $$\DF(\X,\L) =  \frac{b_0a_1 - b_1a_0}{a_0},$$ 
\item  \emph{inner product} of $\alpha$ and $\beta$ to be $$\langle \alpha,\beta \rangle = \frac{d_0 - b_0^2}{a_0},$$
\item $L^2$\emph{-norm} to be $\|(\X,\L)\|_2 = (\langle\alpha,\alpha\rangle)^{1/2},$
\item \emph{relative Donaldson-Futaki invariant} to be $$\DF_{\beta}(\X,\L) = \DF(\X,\L) - \frac{\langle \alpha,\beta \rangle}{r\langle\alpha,\alpha\rangle}\opF(\beta).$$
\end{enumerate} Here we use the notation $\opF(\beta)$ for the Donaldson-Futaki invariant of the product test configuration for $(X,L)$ induced by $\beta$. 
\end{definition}

We shall always assume our test configurations are non-degenerate, in the sense that $\|(\X,\L)\|_2>0$, following \cite{filt}. When $(X,L)$ is a normal variety, this  is equivalent to restricting to test configurations with normal total space \cite{RD-imrn,BHJ}.

\begin{remark}\label{wang}The value $F(\beta)$ equals ($r^{-1}$ times) the corresponding value for the product test configuration induced by $\beta$ for $(\X_0,\L_0)$. This was perhaps first observed in this context by Wang \cite[Theorem 26]{wang1}. This can also be obtained by analytic means in the smooth situation \cite[Proposition 4.3]{RD-kahler}, or by describing the weight as an Euler characteristic \cite{SD2,wang2,YO}, and using constancy of Euler characteristics in flat families. The latter argument extends to imply the norm of $\beta$ can also be calculated either on $X$ or $\X_0$. \end{remark}

\begin{definition}\cite{SD2} We say that $(X,L)$ is \emph{K-semistable} if $\DF(\X,\L)\geq 0$ for all test configurations.\end{definition}

The definition of K-polystability is slightly more delicate, as one should consider test configurations which are isomorphic to products away from codimension two subschemes \cite{stoppa-note}. Supposing now that $(X,L)$ admits an additional $\C^*$-action $\beta$, the definition of relative K-semistability is similar. 
\sloppy

\begin{definition}\cite{gabor-blms} We say that $(X,L)$ is \emph{K-semistable relative to} $\beta$ if $ \DF_{\beta}(\X,\L) \geq 0$ for all $\beta$-equivariant test configurations.\end{definition}

Relative K-polystability can be defined analogously, though we will not need this notion. We now pass from $\C^*$-actions $\alpha,\beta$ to rational elements of the Lie algebra $v_{\alpha},v_{\beta} \in \Lie(\Aut(\X_0,\L_0))$ for notational convenience. The definitions of the quantities involved immediately imply the following.

\begin{lemma}\label{linearity}
The Donaldson-Futaki invariant computed on $(\X_0,\L_0)$ is an additive function on the rational points of the Lie algebra. 
\end{lemma}

Here adding integral elements of the Lie algebra corresponds to composing the (commuting) one-parameter subgroups $\alpha,\beta$. The various quantities involved scale linearly, hence the definitions naturally extend to rational elements of the Lie algebra.  One can in fact extend the definitions to arbitrary, possibly irrational, elements \cite[Theorem 4.14]{cs}, but it will suffice for our purposes to consider rational elements. With this definition, it is clear that one can assume $ \langle \alpha,\beta \rangle = 0$ in the definition of relative K-semistability. 

\section{Proof}

Following the notation in the Introduction, we have a scheme $(X,L)$ with a test configuration $(\X,\L)$ with central fibre $(Y,H)$ and $\C^*$-action induced by $v$. The definition of relative K-semistability requires us to show that for all $v$-invariant test configurations $(\Y,\H)$ for $(Y,H)$ with $\C^*$-action $v_{\Y}$, we have $\DF_v(\Y,\H) \geq 0$. The main tool will be the following result of Li-Wang-Xu and Li-Xu which induces a test configuration for $(X,L)$ from $(\Y,\H)$.

\begin{lemma}\label{lx}\cite[Lemma 3.1]{lwx2}\cite[Lemma 6.1]{LX} In this situation, for any $m \gg 0$ there is a test configuration $(\scZ,\L_{\scZ})$ for $(X,L)$ with central fibre $(\Y_0,\H_0)$ and with $\C^*$-action generated by $mv+v_Y$. \end{lemma}

We are now ready to prove Theorem \ref{intromainthm}. 

\begin{proof}[Proof of Theorem \ref{intromainthm}] Our hypothesis is that $(\X,\L)$ minimises the normalised Donaldson-Futaki invariant over all test configurations for $(X,L)$. In particular we have $$\frac{\DF(\scZ,\L_{\scZ})}{\|(\scZ,\L_{\scZ})\|_2} \geq \frac{\DF(\X,\L)}{\|(\X,\L)\|_2}.$$ 

We may assume $ \langle v,v_{\Y} \rangle = 0$, and hence we wish to show that $\DF(\Y,\H) \geq 0$. This quantity is simply the Futaki invariant of $v_{\Y}$ computed on $\Y_0$. By Remark \ref{wang} the Futaki invariant $F(v)$ in the definition of the relative Donaldson-Futaki invariant can also computed on $(\Y_0,\H_0)$, and in particular all relevant quantities can be calculated on $(\Y_0,\H_0)$.

Suppose for contradiction that $\DF(\Y,\H) = F(v_{\Y})<-\epsilon$ for some $\epsilon\in \Q_{>0}$. By the linearity of the Donaldson-Futaki invariant given by Lemma \ref{linearity}, we have $$\frac{\DF(\scZ,\L_{\scZ})}{\|(\scZ,\L_{\scZ})\|_2}  = \frac{F(v_{\Y}) + mF(v)}{\|v_Y + mv\|_2} < \frac{mF(v)  - \epsilon }{\|mv_{\Y} + v\|_2},$$ where $\|mv_{\Y} + v\|_2$ denotes the norm of the $\C^*$-action induced by $mv_Y + v$ on $(\Y_0,\H_0)$. 
 
We first assume that $F(v) \neq 0$, so that we may scale the actions in such a way that $F(v) = -\|v\|^2_2$. Indeed   since $(\X,\L)$ is an optimal degeneration, its Donaldson-Futaki invariant must be non-positive, and we are assuming $(\X,\L)$ is non-degenerate, so that  $\|v\|_2>0$. This is a standard scaling of the extremal vector field in the study of extremal metrics \cite[p903]{stoppa-szekelyhidi}. Thus the optimality hypothesis implies $$-\|v\|_2  = \frac{F(v)}{\|v\|_2}< \frac{mF(v)  - \epsilon }{\|v_Y + mv\|_2} = \frac{-m\|v\|^2_2 - \epsilon}{{\|v_Y + mv\|_2} }, $$ hence $$-\|v\|_2 \|v_{\Y} + mv\|_2 < -m\|v\|^2_2 - \epsilon.$$ Rearranging and squaring gives $$\|v_Y + mv\|^2_2 > m^2\|v\|_2^2 + 2m\epsilon\|v\|_2 + \frac{\epsilon^2}{\|v\|_2}.$$ Since we assumed  $\langle v,v_{\Y} \rangle=0$, we have $\|v_{\Y} + mv\|^2_2 = m^2 \|v\|_2^2 + \|v_{\Y}\|_2^2.$ Thus $$\|v_{\Y}\|_2^2 > 2m\epsilon\|v\|_2 + \frac{\epsilon^2}{\|v_{\Y}\|_2},$$ which is a contradiction for $m\gg 0$. 

When $F(v)=0$ the argument is easier. Then $(X,L)$ is K-semistable, and the equality $\DF(\scZ,\L_{\scZ}) = mF(v_Y) = m\DF(\Y,\H)$  implies $(Y,H)$ is also K-semistable. 

Hence in both cases $\DF(\Y,\H) \geq 0$ and $(Y,H)$ is K-semistable relative to $v$, as desired.\end{proof}

\bibliography{semistable}
\bibliographystyle{amsplain}

\end{document}